\documentclass[12pt]{amsart}
\usepackage{amsmath,amssymb,amsfonts,a4wide}
\usepackage{mathrsfs,euscript}
\usepackage{xcolor}
\usepackage{hyperref}
\hypersetup{
    colorlinks=true, 
    linktoc=all,     
    linkcolor=blue,  
}
\usepackage{url}

\def\sideremark#1{\ifvmode\leavevmode\fi\vadjust{\vbox to0pt{\vss
\hbox to 0pt{\hskip\hsize\hskip1em%
\vbox{\hsize2cm\tiny\raggedright\pretolerance10000%
\noindent {\color{red}{#1}}\hfill}\hss}\vbox to8pt{\vfil}\vss}}}%

\theoremstyle{plain}
\newtheorem{propn}{Proposition}[section]
\newtheorem{thm}[propn]{Theorem}
\newtheorem{lemma}[propn]{Lemma}

\theoremstyle{definition}
\newtheorem{defn}[propn]{Definition}

\theoremstyle{remark}
\newtheorem{rem}{Remark}

\numberwithin{equation}{section}

\begin{document}

\title{ Quasi-Einstein metrics on sphere bundles}
\author{Solomon Huang}
\address{Department of Mathematics, California State University Fullerton, 800 N. State College Bld., Fullerton, CA 92831, USA.}
\bigskip

\email{solomonhuang0703@csu.fullerton.edu}
\author{Tommy Murphy}
\email{tmurphy@fullerton.edu}
\urladdr{http://www.fullerton.edu/math/faculty/tmurphy/}
\bigskip

\author{Thanh Nhan Phan}
\email{tnp2424@csu.fullerton.edu}

\maketitle

\begin{abstract}
    
We adapt  the work of Hall to find quasi-Einstein metrics on sphere bundles over products of Fano K\"ahler--Einstein manifolds, as well as bundles where only one end is blown down.  
\end{abstract}

\section{Introduction}

In Wang--Wang\cite{ww}, Hermitian--Einstein manifolds were constructed via a Kaluza--Klein bundle ansatz, on (i) $\mathbb{S}^2$-bundles over products of Fano K\"ahler--Einstein manifolds, and (ii) $\mathbb{R}P^2$-bundles over products of Fano K\"ahler--Einstein manifolds.   In subsequent years (i) was generalized in various directions. The most relevant results  for our purposes are to be found in  \cite{wangsurvey}, where the extension of the Wang--Wang results to allow blowdowns  were presented.  Out of this came the results in  \cite{dw}, where a general construction of  Ricci solitons with blowdowns was presented.  

The focus of this paper is on the work of  Hall \cite{hall}  constructing solutions to the  quasi--Einstein metrics on $\mathbb{S}^2$-bundles over products of Fano K\"ahler--Einstein manifolds with blowdowns. The quasi--Einstein equation is a generalization of the Einstein condition.  An important application is that such metrics yield new Einstein metrics.  In particular, for  each quasi-Einstein  metric  one can construct   Einstein metrics on associated families of warped-product manifolds \cite{case}.  The metrics we construct generalize the first family of examples due to Wang--Wang. For compact manifolds, Hall's proof only works if \textit{both} ends are blown down. Our main contribution in this paper is to show Hall's proof can be adapted when one or both ends are not blown down.

To  fix notation and state the main result, we begin with a definition. 
\begin{defn} A compact Riemannian manifold $(M,g)$  is  \textit{quasi--Einstein} if it solves
\begin{equation}
\textrm{Ric}_g + \nabla^2u - \frac{1}{m} du\otimes du + \frac{\epsilon}{2}g = 0,
\end{equation}
where  $u\in C^{\infty}(M)$, $m\in [1, \infty)$ and $\epsilon$ is a constant. 
\end{defn}

Our main result is as follows. We will  explain notation and more concerning the constructions employed in the body of the paper, but we also refer the reader to  \cite{hall}  \cite{ww} for further details. \newpage

Although the setup is quite technical, in essence the symmetry assumptions allow us to reduce the problem to solving systems of ordinary differential equations.  Suppose $(M_i,h_i),\; i = 1, \ldots, r$, are Fano K\"ahler--Einstein manifolds with first Chern classes $c_1(M_i) = p_i\alpha_i$, where  $p_i\in \mathbb{N}$ and the $\alpha_i$ are indivisible classes. Let $M = M_1\times M_2\times\ldots\times M_r$ denote the Riemannian product of $(M_i, h_i)$, $\pi_i: M\rightarrow M_i$ denote the projection onto each factor, and  $p$ denote the $r$-tuple  $(p_1, \ldots p_r)$.   For the $r$-tuple of non-zero integers $q = (q_1, \ldots q_r)$ the integral cohomology class $\sum_{i=1}^r q_i\pi_i^*\alpha_i$ is the Euler class of a principal circle bundle $P_q$ over $M$.    Let $W^{p}_{q}$ denote the associated  $\mathbb{S}^2$-bundle over $M$ coming from the circle action on the Riemann sphere $\mathbb{S}^2$. 

\begin{thm}\label{maint}
With the above notation, 
\begin{enumerate}
    \item[(i)]  if $0 < |q_i|< p_i$ for $i = 1, \ldots r$, then $W^{p}_{q}$ admits non-trivial  quasi-Einstein metrics for any $m>1$, and
    \item[(ii)] if $|q_i|(n_1 + 1) < p_i$ for $2\leq i \leq r$  then there exist  non-trivial  quasi-Einstein metrics for any $m>1$ on the space obtained by blowing $W^{p}_{q}$ down at  the left-hand end and gluing in $\Pi_{m=2}^rM_i$.
    \end{enumerate}
\end{thm}
\noindent\textbf{Acknowledgments} The authors thank the Dept of Mathematics and the College of Natural Science and Mathematics at CSU Fullerton for supporting this research through a summer undergraduate research grant (SH), the deLand Fellowship (TNP), and a Junior Intramural Grant (TM). We thank the referee for  detailed and very helpful suggestions to improve our paper and Stuart Hall for many conversations related to this work. 

\section{Deriving the Quasi-Einstein equations}
We summarize the framework presented in \cite{hall}, \cite{ww}, where various changes of coordinates are performed to simplify the equations. The reader is referred to these works for further details concerning the objects in the construction. In essence, the manifold is of cohomogeneity one, which means there is an isometric group action whose principal orbits have codimension one.  The principal orbits, $P_q$, are circle bundles over the product manifold $M$ and are equipped with a  $U(1)$-connection $\theta$ with curvature $\sum_{i=1}^r q_i \pi_i^*\eta_i$, where $q_i\in \mathbb{Z}$ and $\eta_i$ is the K\"ahler form of $(M_i,h_i)$. The union of all principal orbits is denoted  $M_0 = (0,l)\times P_q$.  On $M_0$ we take the metric 
$$
g = dt^2 + f^2(t)\theta\otimes \theta + \sum_{i=1}^r g_i(t)^2 \pi^*h_i. 
$$
As we will focus on closed manifolds,  the key issue will be how to compactify this metric at $t=0$ and $t=l$.

\begin{lemma}
 Then the quasi-Einstein equations on $M_0$ are given by 
\begin{eqnarray}
    \frac{\ddot f}{f}+\sum _{i=1}^r 2n_i\frac{\ddot g_i}{g_i}+m\frac{\ddot v}{v}=\frac{\epsilon}{2}, \\
    \frac{\ddot f}{f}+\sum_{i=1}^r\left(2n_i\frac{\dot f\dot g_i}{fg_i}-\frac{n_iq_i^2}{2}\frac{f^2}{g_i^4}\right)+m\frac{\dot f\dot v}{fv} = \frac{\epsilon}{2}, \\
    \frac{\ddot g_i}{g}-\left(\frac{g_i}{g_i}\right)^2+\frac{\dot f\dot g_i}{fg_i}+\sum_{j=1}^r2n_i\frac{\dot g_i\dot g_j}{g_ig_j}-\frac{p_i}{g_i^2}+\frac{q_i^2f^2}{2g_i^4}+m\frac{\dot g_i\dot v}{g_iv}=\frac{\epsilon}{2}.
\end{eqnarray}
\end{lemma}
Here  $\cdot$ denotes the derivative with respect to $t$. Quasi-Einstein metrics satisfy a well-known additional constraint: 
\begin{lemma}[Kim-Kim \cite{kk}]
There exists a constant $\mu$ so that
\begin{align}
    v\ddot v+v\dot v\left(\frac{\dot f}{f}+\sum_i2n_i\frac{\dot g_i}{g_i}\right)+(m-1)\dot v^2-\frac{\epsilon}{2}v^2 = \mu.
\end{align}
\end{lemma}

It is now standard to apply  the following change of coordinates on this system of differential equations. Let $s$ be the coordinate on $I=(0,l)$ such that $ds=f(t)dt,\;\alpha(s)=f^2(t),\;\beta_i(s)=g_i^2(t),\;\phi(s)=v(t)$ and $V=\Pi _{i=1}^{i=r}g_i^{2n_i}(t)$.  We choose our antiderivative so that $s$ ranges over the inverval $[0, s_*]$. 
\begin{propn}
Equations (2.1)--(2.4) transform under this change of coordinates into 
\begin{align}
    \frac{1}{2}\alpha '' + \frac{1}{2}\alpha' (\log V)'+\alpha \sum_{i=1}^rn_i\left(\frac{\beta''_i}{\beta_i}-\frac{1}{2}\left(\frac{\beta_i'}{\beta_i}\right)^2\right)+m\left(\frac{\alpha\phi''}{\phi}+\frac{\alpha'\phi'}{2\phi}\right)=\frac{\epsilon}{2}, \\
     \frac{1}{2}\alpha '' + \frac{1}{2}\alpha' (\log V)'-\alpha \sum_{i=1}^r\frac{n_iq_i^2}{2\beta_i^2}+m\frac{\alpha'\phi'}{2\phi}=\frac{\epsilon}{2},\\
    \frac{1}{2}\frac{\alpha'\beta_i'}{\beta_i}+\frac{1}{2}\alpha\left(\frac{\beta_i''}{\beta_i}-\left(\frac{\beta_i'}{\beta_i}\right)^2\right) + \frac{1}{2}\frac{\alpha\beta_i'}{\beta_i}\left(\log V\right)' -\frac{p_i}{\beta_i}+\frac{q_i^2\alpha}{2\beta_i^2}+m\frac{\alpha}{2}\frac{\beta_i'\phi'}{\beta_i\phi}=\frac{\epsilon}{2} \\
    \phi\left(\phi''\alpha+\frac{\phi'\alpha'}{2}\right)+\phi\phi'\left(\frac{\alpha'}{2}+(\log V)'\alpha\right)+(m-1)(\phi')^2\alpha-\frac{\epsilon}{2}\phi^2=\mu.
\end{align}
\end{propn}
Here $'$  denotes the derivative with respect to $s$. For the manipulations which follow,  is helpful to observe that 
$$
\sum_{i=1}^rn_i\frac{\beta_i'}{\beta_i}= \sum_{i=1}^rn_i(\log\beta_i)' = \left(\sum_{i=1}^r\log\beta_i^{n_i}\right)' = 
       \left(\log (\Pi_{i=1}^r g_i^{2n_i})\right)'= (\log V)'.
$$
Equating Equations (2.5) and (2.6) we obtain 
\begin{align}
    -m\frac{\phi''}{\phi}=\sum_{i=1}^{r}n_i\left(\frac{\beta''_i}{\beta_i}-\frac{1}{2}\left(\frac{\beta_i'}{\beta_i}\right)^2+\frac{q_i^2}{2\beta_i^2}\right).
\end{align} Following the standard ansatz we look for solutions that satisfy 
\begin{equation}\label{ansatz}
\frac{\beta''_i}{\beta_i}-\frac{1}{2}\left(\frac{\beta_i'}{\beta_i}\right)^2+\frac{q_i^2}{2\beta_i^2}=0.
\end{equation}
This forces $\phi$ to be linear, so set $\phi(s)=\kappa_1(s+\kappa_0)$ for some constants $\kappa_0,\kappa_1\in\mathbb{R}.$ Substituting $\phi$ into Equation (2.8) yields \begin{align}
    \alpha'+\left((\log V)'+\frac{(m-1)}{(s+\kappa_0)}\right)\alpha=\frac{\epsilon(s+\kappa_0)}{2}+\frac{\mu}{\kappa_1^2(s+\kappa_0)}.
\end{align}
It also easily follows from Equation (2.10) that $\beta_i$ is either quadratic or linear. As explained in \cite{hall}, consistency conditions between the various equations actually force  $\beta_i$ to be of the following form:
 $$\beta_i=A_i(s+\kappa_0)^2-\frac{q^2_i}{4A_i}.$$ Substituting this  into Equation (2.7); $$\alpha'+ \left((\log V)'+m(\log \phi)'-\frac{1}{s+\kappa_0}\right)\alpha =\frac{\epsilon}{2}(s+\kappa_0)+\frac{E}{s+k_0}$$ where $$E:=\frac{8A_ip_i-\epsilon q_i^2}{8A_i^2}.$$ Comparing with Equation (2.11), we see that for the solution to be consistent we must have $$\frac{\mu}{\kappa_1^2}=E=\frac{8A_ip_i-\epsilon q_i^2}{8A_i^2}.$$ Solving for $\alpha$ we observe that  notice that we have a first-order linear differential equation.  Using the integrating factor
\begin{align*}
I&= \textrm{exp}\left(\int (\log V)'+\frac{(m-1)}{s+\kappa_0}ds\right)\\
&=\textrm{exp}\left( \log V+(m-1)\log (s+\kappa_0)\right)\\
&=V(s+\kappa_0)^{m-1},
\end{align*}
the solution is 
$$\alpha(s)=V^{-1}(s+\kappa_0)^{1-m}\int_0^sV(r+\kappa_0)^{m-2}\left(E+\frac{\epsilon}{2}(r+\kappa_0)^2\right)dr.$$
\section{Extending to a smooth metric on a compact manifold}

The conditions $\alpha(0) = \alpha(s_*) = 0$ must be satisfied at both ends, since we wish to collapse the $\mathbb{S}^1$ fibre.  Assuming $\beta_i>0$ on $[0,s_*]$, a smooth collapse of the circle fibre yields a $\mathbb{S}^2$ bundle.  This corresponds to the original approach of \cite{ww}.  This requires that $\alpha'(0) = 2$ and $\alpha'(s_*) = -2$. See \cite{petersen} (section 1.4.4)
as well as  \cite{ww} for further discussion of this condition. 

Subsequently (see \cite{dw}) this construction was generalized to allow \emph{blowdowns}, which are more complicated singularities.   For blowdown at the left end (where $s=0$), we require that $M_1 = \mathbb{C}P^{n_1}$ is complex projective space with the Fubini Study metric. The vanishing of this factor simultaneously with the circle fiber yields, via the Hopf fibration, a sphere $\mathbb{S}^{2n+1}$ in the normal space.  To do so smoothly, we require $\beta_1(0) = 0$ and $\beta_1'(0) = 1$. See \cite{dw} for a discussion of this, especially  following Equation (4.17). In a similar way, to blowdown smoothly at the right-hand end (where $s=s_*$), we require $M_r = \mathbb{C}P^{n_2}$ is a complex projective space equipped with the Fubini--Study metric, and  $\beta_r(s_*) = 0$ and $\beta_r'(s_*) = -1$. In this more general framework, the case with no blowdowns corresponds to $n_1 = n_r = 0$. Here we just collapse the $\mathbb{S}^1$-fiber at both ends, which is of course the setting of  $\mathbb{S}^2$-bundles already discussed.

\section{Blowdowns in Hall's Proof}\label{blowdowns}
 Our proof follows the standard ansatz and that of \cite{hall}. Note $\epsilon = -1$ in Hall's notation.   In this section we explain the approach of Hall to compute $\kappa_0$ and $s_*$ To compactify $M_0$, Hall  firstly requires 
$$
\alpha(0) = 0, \ \alpha'(0) = 2, \ \beta_1(0) = 0, \ \text{and}\ \beta_1'(0) = 1.
$$
The conditions on $\beta_1$ show that a blowdown of $M_1$ at the left-hand end is assumed/required in Hall's proof. 
Next, 
$$
\beta_1'(0) = 1 \implies 2A_1\kappa_0 = 1  \implies A_1 = \frac{1}{2\kappa_0},
$$
and also 
\begin{align*}
    0 & = \beta_1(0)\\
    &= A_1\kappa_0^2 - \frac{q_1^2}{4A_1} \\
&= \frac{\kappa_0}{2}\left( 1 - q^2_1\right)
\end{align*}
which implies $q_1^2 = 1$. Since $p_1 = n_1 + 1$ we see the normalization conditions become
$$
E = \frac{\mu}{\kappa_1^2} = \frac{\kappa_0}{2}\left(4(n_1 + 1) + \kappa_0\right) = \frac{8A_ip_i + q_i^2}{8A_i^2} \ \text{for } 2\leq i \leq r.
$$
In particular, observe that $\kappa_0$ is determined via feeding $A_1 = \frac{1}{2\kappa_0}$ into the consistency condition, so this number can only be determined in Hall's framework if there is a blowdown at $s=0$. 

Finally, it remains to determine  $s_*$ (the length of the interval).  Hall writes that for the manifold to extend smoothly at the point $s=s_*$ the conditions $q_r=1$, $p_r = n_r+1,$ and $-1 = 2A_r(s_*+ \kappa_0)$ are imposed. As when $s=0$, these conditions result from assuming a non-trivial blowdown occurs at the right-hand end. Plugging these into the consistency conditions
$$
\frac{\kappa_0}{2}\left(4(n_1 + 1) + \kappa_0\right)  = E   = \frac{8A_rp_r + q_r^2}{8A_r^2} 
$$
which yields, following some elementary algebra, that
\begin{equation}\label{hall2}
s_* = \sqrt{\kappa_0(4(n_1 + 1) + \kappa_0) + 4(n_r + 1)^2} - \kappa_0 + 2(n_r + 1),
\end{equation}
as the reader may easily check. 

\begin{rem}
The case with no-blowdowns is important to analyse, since many important examples are obtained in this way. For instance, see Section 3.1   of \cite{hall}.
\end{rem}

Putting $n_1=n_r = 0$, we obtain the following equations which determine  $\kappa_0$ and $s_*$ as functions of $E$:
\begin{equation}\label{hall0}
E = \frac{\kappa_0}{2}\left( 4 + \kappa_0\right) = 2\kappa_0 + \frac{\kappa_0^2}{2}, 
\end{equation}
\begin{equation}\label{halls}
s_* =  \sqrt{\kappa_0(4+ \kappa_0) + 4} - \kappa_0 + 2.
\end{equation}

We choose to not fully simplify Equation (4.3) so the connection with the quadratic equation in the proof of Theorem \ref{maint} is clear.

\section{An Alternative Approach}

The main idea in this work is to find an alternative way of determining $\kappa_0$ and $s_*$ as functions of $E$ which does not depend upon the existence of non-trivial blowdowns. In \cite{ww}, for the Einstein equations the interval on which $\tilde{s}$ is defined, where $\tilde{s} = s+ \kappa_0$, is determined by showing both endpoints of the interval can be interpreted as roots of a certain quadratic. It turns out this also holds in the setting of the quasi-Einstein equations.

\begin{lemma} With the same notation as above; 
\begin{enumerate}
    \item If $\alpha(0) = 0$, then  $\kappa_0$ is a root of the quadratic $\frac{1}{2}x^2 + 2x - E.$
    \item  If $\alpha(s_*) = 0$, then  $-\kappa_0 - s_*$ is a root of the quadratic $\frac{1}{2}x^2 + 2x - E.$
\end{enumerate}
\end{lemma}
\begin{proof}
Assume $\alpha(0) = 0$ and noting, as per \cite{ww}, this directly implies $\alpha'(0) = 2$,  feed this into Equation (2.7) to obtain
$$
\frac{-p_i + \beta_i'}{\beta_i}(0) = -\frac{1}{2}.
$$
Rewriting our constraint as 
$$
p_i = A_iE - \frac{q_i^2}{8A_i}
$$
and, noting that $\beta_i'(0) = 2A_i\kappa_0$ and $\beta_i(0) = A_i\kappa_0^2 - \frac{q_i^2}{4A_i}$ transforms this equation to
$$
-EA_i  + 2A_i\kappa_0 = \frac{-1}{2}A_i\kappa_0^2
$$
and so $\kappa_0$ is a root of the quadratic as claimed.  The proof for the second claim is analogous, simply adjusting so that $\alpha'(s_*) = -2$, and hence is  left to the reader. 
\end{proof}

\section{The proof of Theorem \ref{maint}}

\begin{proof} (i)  We simply have to show the roots of the quadratic $\frac{1}{2}x^2 + 2x - E$ yield the same expressions for $\kappa_0$ and $s_*$ as those coming from
  Equations (4.2) and (4.3).  Then the rest of the proof in \cite{hall} goes directly through, and thus the existence of quasi-Einstein metrics on $\mathbb{S}^2$ bundles is established. 
  
 Applying the quadratic formula, we see $\kappa_0$ is  (the larger)  root of $\frac{1}{2}x^2 + 2x - E$ if and only if
 $$
\kappa_0 =  \frac{-4 + \sqrt{16 + 8E}}{2}
 $$
 which simplifies to  
 \begin{equation}\label{tm1}
 E = \frac{\kappa_0^2}{2} + 2\kappa_0.
 \end{equation}
 This agrees with Equation (4.2). Similarly, $-s_*- \kappa_0$ corresponds to the other root of the quadratic if, applying the quadratic formula,
\begin{align*}
- s_*-\kappa_0 &=  \frac{-4 - \sqrt{16 + 8E}}{2} \\
&= -2 - \frac{1}{2}\sqrt{16 + 16\kappa_0 +  4\kappa_0^2} \ \ \ \ \ \text{ by Equation (6.1)}\\
&= -2 - \sqrt{4 + 4\kappa_0 +\kappa_0^2} 
\end{align*}
 which  agrees with Equation (4.3). This suffices to establish the theorem.
 \bigskip
 
 (ii) The proof is completely analogous, once one realizes that Lemma 5.1 is independent of the existence of blowdowns.

\end{proof}

\begin{rem} We are forced to choose $\kappa_0$ to be the larger  and $-s_*-\kappa_0$  the smaller root. They cannot be the same root, as if they agreed $s_*=-2\kappa_0$, meaning $s+\kappa_0$ would vanish when $s = -\kappa_0$. This cannot happen as then $\alpha(s)$ would be undefined at this point. Furthermore  $-s_*-\kappa_0$  must be the smaller root to agree with Hall's setup.

As a sanity check, we remark that the length of the interval is 
$$
s_* = s_*+\kappa_0 - \kappa_0 = 2 + \frac{1}{2}\sqrt{16 + 8E} - \left(-2 +\frac{1}{2}\sqrt{16 + 8E}\right) =4.
$$
The formula obtained in \cite{hall} for blowdowns with $n_1=n_r$ is 
$s_* = 4\sqrt{n_1 + 1}$, so putting $n_1=0$ we see our answer agrees with this.
\end{rem}

\end{document}